\documentclass[12pt]{amsart}
\usepackage{amssymb}
\usepackage{amsmath}
\usepackage{amsthm}
\usepackage{graphicx}

\newtheorem{Theorem}{Theorem}[section]
\newtheorem{Example}[Theorem]{Example}
\newtheorem{Corollary}[Theorem]{Corollary}

\newtheorem{Definition}[Theorem]{Definition}
\newtheorem{Remark}[Theorem]{Remark}

\begin{document}
	
	\title[On Transitivities for Skew Products]
	{On Transitivities for Skew Products}

		\author{Nayan Adhikary and Anima Nagar}
	
	\address{Department of Mathematics, Indian Institute of Technology Delhi,
		Hauz Khas, New Delhi 110016, INDIA.}
	
	\email{nayanadhikarysh@gmail.com, anima@maths.iitd.ac.in}
	
	\subjclass[2020]{37B05, 47A16}
	
	\begin{abstract} 
		The dual concepts of `universality' and `hypercyclicity' are better understood and studied as `topological transitivity'. In this article we consider transitivity properties of skew products, essentially with non-compact fibers. We study the `Universality Conditions' and `Hypercyclicity Criterion' associated with  the dynamical properties of transitivity, weakly mixing and mixing for these skew products.
	
	\end{abstract}
	\maketitle
	\smallskip
	\noindent{\bf\keywordsname{}:} {skew products, linear operators, hypercyclicity, topological transitivity, weakly mixing, mixing.}

	\section{Introduction}
	A \emph{topological dynamical system} is a pair $(X,G)$, where $X$ is usually a Hausdorff phase space and $G$ is a topological group or monoid acting on $X.$ If $G$ is a group, the system $(X,G)$ is called a \emph{flow}, whereas if $G$ is a monoid then $(X,G)$ is called a \emph{semiflow}.  When $G=\mathbb{Z},$ there is a generating homeomorphism $T:X\to X$ such that the action is given by the iterates of $T$, resulting in the \emph{cascade} $(X,T).$ On the other hand,  $G=\mathbb{Z}_+$ results in a \emph{semicascade} $(X,T),$ where $T$ is a continuous self map.\par 
		The  dynamical properties  of ``\emph{topological transitivity}" and its several stronger forms like \emph{``weakly mixing''}, \emph{``mixing''}, \emph{``minimality''} form an important study in   topological dynamics.\par 
	Skew product transformations were defined by H. Anzai \cite{Anz} to obtain certain counter examples in ergodic theory. In \cite{ellis}, R. Ellis constructed some minimal discrete flows as skew product extensions of minimal flows.	 In \cite{KS} K. Schmidt outlined a general theory of skew product extensions. Our skew products are essentially based on these ideas in a topological set up. \par 
	
Transitivity properties of skew products comprise an interesting study when their fibers are unbounded manifolds. The question we take up here is -- `when can a skew product with a non-compact fiber be transitive, weakly mixing or mixing?'.	We explore some conditions for topological transitivity of skew products in a topological setup and further investigate these properties  in the context of linear dynamics. Due to the linear structure, such transitivity conditions can be nicely formulated via `Universality' and `Hypercyclicity' conditions. In \cite{bayart} F. Bayart et. al. studied some conditions of topological transitivity for skew products of linear operators.  Here we analyze certain aspects of their conditions and relate them with weakly mixing and mixing properties of such skew products.\par

	Throughout $\mathbb{Z}$, $\mathbb{Z}_+$,  $\mathbb{N}$ and $\mathbb{R}$ stand for the usual space of all integers, non-negative integers, natural numbers and reals respectively; and all spaces in the sequel are assumed to be Hausdorff spaces without isolated points. The word `opene' stands for nonempty and open. \par
	
	\bigskip

	In section 2 we discuss some preliminaries from topological dynamics, measurable dynamics and linear dynamics. In section 3 we  investigate several transitivity properties for  skew products in a topological setup.  Section 4 concerns linear dynamics. We observe that all the topological conditions discussed in the section 3 have an interesting representation due to the underlying linear structure.\par

	\section{Preliminaries}
	
	\subsection{ Topological Dynamics.} We mostly refer to \cite{V, GH} for the basics mentioned here.
	
	 Let $(X,G)$	be a (semi)flow. Note that here, each $g \in G$ can be indentified with an element of $X^X$, where $ g $ can be regarded as a homeomorphism on $X$ if $G$ is a group and as a continuous self map on $X$ if $G$ is a monoid. The study mainly concerns the asymptotic properties of the orbits $\mathcal{O}_G(x) = \{gx: g \in G\}$ for all $x \in X$. For any pair of opene subsets $U$ and $V$ of $X$ we consider the hitting time sets $N_G(U,V)$ or simply $N(U,V)=\{g\in G:g(U)\cap V \neq \emptyset\}$ and for any point $x\in U,$ $N(x,U)=\{g\in G: g(x)\in U \}.$ 
	
	A (semi)flow $(X,G)$ is called \emph{Topologically Transitive} if for every pair of opene sets $U, V \subset X$, there exists $g\in G$ such that $g(U)\cap V \neq \emptyset,$ i.e. $N_G(U,V)$ is nonempty. $(X,G)$ is \emph{Minimal} if for every $x\in X$ the orbit set $\mathcal{O}_G(x)$ is dense in $X$, is called \emph{Weakly Mixing} if the product system $(X \times X, G)$ with the diagonal action of $G$ is topologically transitive, and is called \emph{Mixing or Topologically Mixing} if for every pair of opene  sets $U, V \subseteq X$,  $g(U) \cap V \neq \emptyset$ for every $g$ in the complement of a compact set in $G$.

	\vspace{0.1cm}
	
	Since most of our results are about (semi)cascades, we further explore these dynamical properties for (semi)cascades.\par
	
Recall that a subset $A\subset (\mathbb{Z}_+)\mathbb{Z}$ is called \emph{cofinite} if there exists $m\in \mathbb{N}$ such that $(\mathbb{Z}_+)\mathbb{Z}\setminus \{k: |k| \leq m\}\subset A,$ \emph{thick} if for any $k\in \mathbb{N},$ there exists $n\in (\mathbb{Z}_+)\mathbb{Z}$ such that $\{n,n+1,\dots, n+k\}\subset A$ and \emph{syndetic} if there exists some $M\in \mathbb{N}$ such that for every $n\in (\mathbb{Z}_+)\mathbb{Z},$ $\{n,n+1,\dots, n+M\}\cap A \neq \emptyset.$ \par

\bigskip 	
	
		 The (semi)cascade $(X,T)$ is \emph{Topologically Transitive} if for every pair of opene sets $U, V \subset X$, there exists $n\in \mathbb{N}$ such that $T^n(U)\cap V \neq \emptyset.$ When $X$ is a second countable, Baire space without any isolated points, topological transitivity is generally given by the existence of a point $x_0\in X$ with the orbit $\mathcal{O}_T(x_0) =  \mathcal{O}(x_0)=\{T^n(x_0):n\in (\mathbb{Z}_+)\mathbb{Z}\}$ dense in $X.$ Such a point is called a transitive point of $(X,T)$. The set of all transitive points of $(X,T),$ denoted by $Trans_T$ is a dense $G_{\delta}$ subset of $X.$ Also, $(X,T)$ is topologically transitive if and only if for every pair of opene sets $U$ and $V,$ the set $N(U,V)$ is infinite.
		
		\vspace{.25cm}
		
		 Let $X$ be a compact metric space. A point $x\in X$ is said to be an \emph{almost periodic point} in the (semi)cascade $(X,T)$ if for every open set $U$ containing $x,$ the set $N(x,U)$ is syndetic. In a minimal system $(X,T)$, every point $x\in X$ is an almost periodic point.
		 
		\vspace{.25cm}
		
		 The (semi)cascade $(X,T)$ is \emph{Weakly Mixing} if the product system $(X \times X, T \times T)$ is topologically transitive. Furstenberg's intersection lemma gives that for weakly mixing systems and every $n\in \mathbb{N}$, the product space $(X\times \dots \times X, T\times \dots \times T)$ (n-times) with diagonal action of $T$ is also topologically transitive. Also  for every pair of opene sets $U$ and $V$, the set $N(U,V)$ is thick.
		 
		\vspace{.25cm}
		
		 The (semi)cascade $(X,T)$ is \emph{Mixing or Topologically Mixing} if for every pair of opene  sets $U, V \subseteq X$, there exists an $N \in \mathbb{N}$ such that $T^n(U) \cap V \neq \emptyset$ for all $|n| \geq N$, i.e. the set $N(U,V)$ is cofinite.

		\vspace{.25cm}
		
		Let $X$ be a metric space. Also, let for each $n\in \mathbb{N},$ $T_n:X\to X$ be a continuous map and $T_0$ be the identity map on $X.$  The sequence $(T_n)_{n\in \mathbb{Z_+}}$ is called \emph{commuting} if $T_m\circ T_n=T_n \circ T_m$ for all $m,n\in \mathbb{Z_+}$. We can  define  \emph{universality properties} for the sequence $(T_n)_{n\in \mathbb{Z_+}}$ on $X$. 
		
\begin{Definition} \cite{Erdmann,linear chaos}	
	A point $x\in X$ is called \emph{universal} for $(T_n)$ if the orbit set $\mathcal{O}(x,T_n)=\{T_nx:n\in \mathbb{Z}_+\}$ is dense in $X.$ The sequence  $(T_n)$ is called \emph{topologically transitive} if for every pair of opene $U,V \subset X$ there exists $n\in \mathbb{N}$ such that $T_n(U)\cap V \neq \emptyset.$	\par
	
	The sequence  $(T_n)$ is \emph{weakly mixing} if the sequence $(T_n\times T_n)$ is topologically transitive i.e., for every opene $U_1,U_2,V_1,V_2 \subset X,$ there exists $n\in \mathbb{N}$ such that $T_n(U_1)\cap V_1 \neq \emptyset$ and $T_n(U_2)\cap V_2 \neq \emptyset.$
	
\end{Definition}
	
\begin{Remark}\cite{linear chaos}
		When $X$ is a second countable, Baire space then topological transitivity of $(T_n)$ is equivalent to the existence of a universal point. Moreover the set of all universal points of $(T_n)$ is a dense $G_{\delta}$ subset of $X.$ \par
		
	For every pair of opene sets $U$ and $V$, we consider the hitting time set $N_{(T_n)}(U,V)=\{n\in \mathbb{N}:T_n(U)\cap V \neq \emptyset\}$. If  $(T_n)$ is commuting and weakly mixing then for any finite collection of opene $U_1,\dots U_k,$$V_1,\dots,V_k$, the set $\displaystyle{\bigcap_{i=1}^k} N_{(T_n)}(U_k,V_k)\neq \emptyset.$\par
			  
		Note that unlike the maps, weakly mixing property of sequence $(T_n)$ may not imply that $N_{(T_n)}(U,V)$ is thick for every pair of opene  $U,V \subset X.$
\end{Remark}

	\begin{Remark}
		It is clear that the semicascade $(X,T)$ is topologically transitive or weakly mixing if and only if the sequence $\{T,T^2,T^3,\dots\}$ is topologically transitive or weakly mixing respectively. In such a case, the universal point is called a transitive point. \par 
	\end{Remark}

	\begin{Definition}\cite{Bes2, linear chaos}
		A sequence $(T_n)$ of self maps defined on a metric space $X$ is called \emph{hereditarily transitive} with respect to the sequence $(n_k)$ if for every subsequence $(n_{k_l})$ of $(n_k)$, the sequence of functions $(T_{n_{k_l}})$ is topologically transitive.\par 
		The sequence $(T_n)$ is called\emph{ hereditarily transitive} if there exists an increasing sequence $(n_k)$ of positive integers such that $(T_n)$ is hereditarily transitive with respect to $(n_k)$. 
	\end{Definition}
	
	\begin{Theorem}\label{hereditarily}\cite{Bes2,linear chaos}
		Let $X$ be a second countable metric space. Then a commuting sequence $(T_n)$ of self maps defined on $X$ is weakly mixing if and only if $(T_n)$ is hereditarily transitive.\par 
		In fact, if $(T_n)$ is hereditarily transitive with respect to an increasing sequence $(n_k)$ of positive integers then for every subsequence $(n_{k_l})$ of $(n_k)$, the sequence $(T_{n_{k_l}})$ is weakly mixing.
	\end{Theorem}

	 \subsection{Measurable Dynamics.} We recall some relevant definitions and results from measurable dynamics for which we mostly refer to  \cite{glasner, walters}. For a compact metric space $A,$ let $\mathcal{B}(A)$ or $\mathcal{B}$ be the $\sigma$-algebra of all Borel subsets of $A$ and $\mu$ be a Borel probabilty measure on $\mathcal{B}(A)$. We recall that this measure has full support if $\mu(U)>0$ for every opene $U\subset A.$ In the probability space $(A,\mathcal{B},\mu)$, a transformation $f:A\to A$ is said to be a \emph{measurable transformation} if for every $B\in \mathcal{B},$ the set $f^{-1}(B)\in \mathcal{B}$. Also, $f$ is called \emph{measure-preserving} if it is measurable and $\mu(f^{-1}(B))=\mu(B)$ for every $B\in \mathcal{B}.$ Recall that a measure-preserving system $(A,\mathcal{B},\mu, f)$ or shortly $(A,\mu,f)$ is said to be 
	\emph{ergodic}  if whenever $\Gamma$ is a strictly $f$-invariant measurable set $(\text{i.e}\ f^{-1}(\Gamma)=\Gamma)$ then either $\mu(\Gamma)=0$ or $\mu(\Gamma)=1.$ From Krylov-Bogolyubov theorem \cite{kb}, we know that for any compact metric space $A$ and a continuous function $f:A\to A$ there exists a $f$-invariant probability measure $\mu$ on $A$, which is also ergodic. We say that $f$ is \emph{uniquely ergodic} if  there exists exactly one $f$-invariant probability measure $\mu$ on $A$. A uniquely ergodic system $(A,\mu, f)$ is called \emph{strictly ergodic}  if $\mu$ has full support \cite{furs, oxotoby}. Since the support of an invariant measure is a closed set, we can conclude that the uniquely ergodic system $(A, \mu ,f)$ is minimal if and only if $\mu$ has full support on $A$. This again provides a nice connection of measurable and topological dynamics. 
	\par
	The system $(A,\mu,f)$   is \emph{weakly mixing} if for every pair of measurable sets $A$ and $B$
	\begin{center}
		$\displaystyle{\lim_{n\to \infty}}\frac{1}{n}\displaystyle{\sum_{i=0}^{n-1}}|\mu(f^{-i}(A)\cap B)-\mu (A)\mu (B)|=0$
	\end{center} 
		Analogous to the topological case it is known that in a probability space $(A,\mathcal{B},\mu)$ a measure-preserving transformation $f:A\to A$ is weakly mixing if and only if $f\times f$ is ergodic with respect to the product measure induced from $\mu$.
	\begin{Theorem}(Birkhoff Ergodic Theorem)\label{Birkhoff}\cite{walters}
		Let the system $(A,\mu,f)$ be ergodic, where $A$ is a compact metric space and $\mu$ is a Borel probability measure. Then for every $\phi \in L^1(\mu),$ and for $\mu$-almost every $a\in A,$\par 
		\begin{center}
			$\displaystyle{\lim_{n\to \infty}}\frac{1}{n}\displaystyle{\sum_{i=0}^{n-1}}\phi(f^ia)=\displaystyle{\int_A} \phi\  d\mu$ 
		\end{center} 
	\end{Theorem}
	\begin{Theorem}\label{oxtoby}(Oxtoby's Theorem)\cite{oxotoby}
		Let the system $(A,\mu,f)$ be uniquely ergodic, where $A$ is a compact metric space. Also let $\phi$ be a real valued continuous function defined on $A.$ Then for every  $a\in A,$  $\frac{1}{n}\displaystyle{\sum_{i=0}^{n-1}}\phi(f^ia) \to \displaystyle{\int_A} \phi\  d\mu$ uniformly as $n\to \infty.$  
		
	\end{Theorem}

	\vspace{0.1cm}
	
	\subsection{Linear Dynamics} Let $X$ be an infinite dimensional separable Banach space and $\mathcal{L}(X)$ denote the set of all bounded linear operators on $X$. Then $T\in \mathcal{L}(X)$ is \emph{hypercyclic} if there exists a vector $x\in X$ such that the orbit $\mathcal{O}(x)$ under $T$ is dense in $X.$ In \cite{birkhoff}, Birkhoff constructed the first known example of a hypercyclic operator. We  refer to \cite{linear,linear chaos} for more details on dynamics of hypercyclic operators. For an infinite dimensional separable Banach space, the Baire category arguments conduce  that transitivity here is the same as hypercyclicity. Weakly mixing  of the linear system $(X,T)$  is nicely presented by  what is known as `Hypercyclicity Criterion'.

	\begin{Definition}\cite{linear,Bes}
		Let $X$ be an infinite dimensional separable Banach space and $T\in \mathcal{L}(X)$. Then we say that $T$ satisfies the \textit{Hypercyclicity Criterion} with respect to the increasing sequence $(n_k)$ of natural numbers if there exist two dense subsets $D_1, D_2$ of $X$ and a sequence of mappings $S_{n_k}:D_2 \to X$ such that\par 
		$(i)$ $T^{n_k} x \to 0$ for every $x\in D_1.$ \par 
		$(ii)$ $S_{n_k} y \to 0$ for every $y\in D_2.$ \par
		$(iii)$ $|| T^{n_k}S_{n_k}y - y|| \to 0$ as $k\to \infty$ for every $y\in D_2.$ 
	
	We say that $T$ satisfies Hypercyclicity Criterion if there exists an increasing sequence $(n_k)$ of positive integers such that $T$ satisfies Hypercyclicity Criterion with respect to this sequence.
\end{Definition}
\begin{Theorem}(Bès - Peris Theorem \cite{Bes2})\\
	Let $X$ be an infinite dimensional separable Banach space and $T\in \mathcal{L}(X)$.. Then $(X,T)$ is weakly mixing if and only if $T$ satisfies Hypercyclicity Criterion.\par 
	Moreover, If $T$ satisfies Hypercyclicity Criterion with respect to the full sequence $(n)$ then $(X,T)$ is mixing.
	
\end{Theorem}
	
	The Hypercyclicity Criterion can also be considered for sequence of linear continuous operators. In this case it is called `Universality Criterion'.
	
	\begin{Definition}\cite{linear chaos}
		Let $X$ be an infinite dimensional separable Banach space and $(T_n)$ be a sequence of bounded linear operators. Then we say that $(T_n)$ satisfies \emph{Universality Criterion} with respect to some increasing sequence $(n_k)$ if there exist two dense subsets $D_1, D_2$ of $X$ and a sequence of mappings $S_{n_k}:D_2 \to X$ such that\par 
		$(i)$ $T_{n_k} x \to 0$ for every $x\in D_1.$ \par 
		$(ii)$ $S_{n_k} y \to 0$ for every $y\in D_2.$ \par
		$(iii)$ $|| T_{n_k}S_{n_k}y - y|| \to 0$ as $k\to \infty$ for every $y\in D_2.$ 
	\end{Definition}
	
	The Bès - Peris theorem \cite{Bes2} also holds for commuting sequence of operators (see \cite{linear chaos}), i.e., a commuting sequence $(T_n)$ of operators is weakly mixing if and only if $(T_n)$ satisfies Universality Criterion. Moreover, in this case these two conditions are equivalent with hereditarily transitivity of the sequence $(T_n)$ of linear operators.

\bigskip

	\section{Topological Transitivity of Skew Products for (Semi)Flows}
	
We recall

	\begin{Theorem}(Rohlin's Theorem for skew products)\cite{glasner}
	Let $X\to Y$ be a factor map of dynamical systems with $X$ ergodic, then $X$ is isomorphic to a skew product over $Y.$ Explicitly, there exists a standard probability space $(U, \mathcal{U}, \rho)$ and a measurable cocycle $\alpha: \Gamma \times Y \to Aut(U,\rho)$ with $X \cong Y \underset{\alpha}{\times} (U, \rho) = (Y\times U, y\otimes \mathcal{U}, \nu \times \rho, \Gamma ),$ where $\gamma(y,u)=(\gamma y, \alpha(\gamma ,y)u).$
\end{Theorem}
	
We note that an exact topological analogue of the same is not possible. One can refer to \cite{ellis} for  examples and counter-examples. Hence the only analogy that can be achieved topologically is to determine the dynamical properties that are carried over to the skew product extensions.

In (\cite{GH}, Chapter 14) the skew products are defined as cylinder sets. For a self homeomorphism $\theta$ on a topological space $Y$, a map $f: Y \to \mathbb{R}$ and $X:= Y \times \mathbb{R}$ the homeomorphism $\phi: X \to X$ defined as $\phi(y,r):= (\theta(y), r +f(y) )$ defines a skew product. Further conditions for $\phi$ to be transitive are derived. In \cite{A}, transitivity of these cylinder transformations is studied when $f: Y \to \mathbb{R}^n$ for some $n \in \mathbb{N}$. 

In \cite{GW} it is shown that for compact metric space $ Y $, there exists minimal and weakly mixing homeomorphisms on $ X := Z \times Y $, of
the form $ (z,y) \to (\sigma z, h_z(y)) $, where $ (Z, \sigma) $ is an
arbitrary compact metric minimal flow and $ z \to h_z $ is a continuous map from $ Z $ to the space of homeomorphisms of $ Y $.

Building on this theory, for a locally compact group $G$, $ T $ a homeomorphism of a locally compact topological Hausdorff space $X$, a continuous $\phi: X \to G$ and the skew product $T_\phi: X \times G \to X \times G$ defined as $T_\phi(x,g) := (Tx, \phi(x)g) $, the conditions leading to topological transitivity of $(X \times G, T_\phi)$  is studied in \cite{LM}. Furthermore, a study of transitivity of such skew products when $G =\mathbb{R}$ is made in \cite{G}.

\bigskip

Thus we see that all these studies were made for  skew products extensions with locally compact or especially real fibers. We extend such a study for unbounded, not necessarily locally compact fibers.

\subsection{Skew Products (semi)flows}   There are various definitions of skew products in literature \cite{bayart,costakis,GW,GH,LM}. For our work we consider the definition of skew product from \cite{bayart}, infused with the ideas from \cite{GW} with some modifications. Though  most  of the other representations of skew products can be deduced from this definition.
	
	\begin{Definition}\label{skew prod gen}
		Let $(A,G)$ be a flow, where $A$ is a  Hausdorff space and $(G,*)$, with identity $e$, is a topological group  acting on $A.$ Also, let $X$ be a Hausdorff space and $(H, \diamond)$, with identity $\breve{e}$, be a topological group acting on $X$ giving another flow $(X,H)$.\par 
		Let $\phi : G\times A \to H$ be a continuous map satisfying the identities:
		\begin{enumerate}
			\item $\phi(e,a)=\breve{e}, a\in A$
			\item $\phi(g,g^{\prime}a) \diamond \phi(g^{\prime},a)=\phi(g*g^{\prime},a)$ for $g,g^{\prime}\in G$ and $a\in A.$
		\end{enumerate} 
		 Such a $\phi$ is called a \emph{cocycle}. \par 
		 The cocycle $\phi$ generates a group $S_G^{\phi}=\{S_g^{\phi}:g\in G\}$ with $S_g^{\phi}:A\times X \to  A\times X$ defined as $S_g^{\phi}(a,x)=(ga,\phi(g,a)x)$ for $g\in G, a\in A$ and $x\in X.$ The flow $(A\times X, S_G^{\phi})$ thus defined is called a \emph{skew product}. Here $(A,G)$ is the \emph{base} and $(X,H)$ is the \emph{fiber} of the skew product.\par
		The cocycle $\phi$ is called a \emph{coboundary} if there exists a continuous function $ F: A \to H $  such that  $ \phi(g,a) := F(ga) \diamond {F(a)}^{-1} $.
		
			\end{Definition}

			\begin{Remark}
				We note that $S_G^{\phi}$ forms a group under the composition as $(S_g^{\phi}\circ S_{g^{\prime}}^{\phi})(a,x)=S_g^{\phi}(g^{\prime}a,\phi(g^{\prime},a)x)=((g*g^{\prime})a,\phi(g*g^{\prime},a)x)$ with identity $S_{e}^{\phi}$. Also,  inverse here can be defined as $(S_g^{\phi})^{-1}=S_{g^{-1}}^{\phi}$ for every $g\in G.$\par 
				In Definition \ref{skew prod gen}, if we take $H$ to be the homeomorphism group on $X$ then it gives the skew product defined in \cite{GW}. 
			\end{Remark}
			
			\begin{Example}
			We consider  Furstenberg's example of skew product on a torus.\par Consider   the unit circle $A = \mathbb{S}^1$ bijective with $[0,1)$. Then $\mathbb{S}^1$ is also a topological group with the operation  of addition modulo $1$. Let $ f(a)= a+\alpha (\mod 1) $  where $ \alpha \in (0,1)$ is an irrational number. 
			In our Definition \ref{skew prod gen} take $G=\mathbb{Z},$ $H=\mathbb{S}^1$ and the cocycle $\phi: \mathbb{Z} \times \mathbb{S}^1 \to \mathbb{S}^1$ as 
				\begin{center}
				$\phi(n,a) = \begin{cases} \sum \limits_{k=0}^{n-1}  f^k(a) \ \ \text{if} \ \ n \geq 1,\\
					0 \ \ \text{if} \ \ n=0,\\
				 -\sum \limits_{k=1}^{-n}  f^{-k}(a)	 \ \ \text{if} \ \ n <0. \end{cases}$
			\end{center}
			
			This gives a skew product $F:\mathbb{S}^2 \to \mathbb{S}^2$ as $F(a,x)=(a+\alpha,a+x) (\mod 1)$ for every $(a,x)\in \mathbb{S}^2.$ Then for every $n\in \mathbb{Z},$ $F^n(a,x)=(f^n(x),\phi(n,a)+x)$ for $(a,x)\in \mathbb{S}^2.$ Here the base is the irrational rotation $(\mathbb{S}^1,f)$ and the fiber is the flow $(\mathbb{S}^1,H)$.

			\end{Example}
		
		
		\begin{Remark}
			The base of a skew product can also be taken to be an ergodic transformation $(A, \mathcal{B}, \mu,f)$, where the probability measure $\mu$ has full support on a compact metric space $A$.
		\end{Remark}
			
			\begin{Remark}
				In Definition \ref{skew prod gen}, we can also assume that $G$ and $H$ are topological monoids or only $G$ is topological monoid. In such a case $ S_G^{\phi}$ is a topological monoid defining the skew product semiflow $(A\times X, S_G^{\phi})$.
			\end{Remark} 
		
			\begin{Theorem}
			Let $(A,G)$ and  $(X,H)$ be two mixing flows, where $A$  and $X$ are infinite Hausdorff spaces and the topological groups $G$ and $H$ are infinite with $H$ being metrizable. \par 
			 Let $\phi: G \times A \to H$ be a cocycle defining the skew product flow $(A \times X, S^\phi_G)$. \par 
			If for every  opene $W \subset A$ and compact $K \subset G$, the set $\{\phi(g,a): a \in W, g \in G \setminus K\}$ is an unbounded subset of $H$     then $(A \times X, S^\phi_G)$ is topologically transitive.
		\end{Theorem}
		\begin{proof}
			Let $A_1, A_2 \subset A$ and $U_1, U_2 \subset X$ be opene. Then there exists   compact $G_1 \subset G$ and $H_1 \subset H$ such that $g(A_1) \cap A_2 \neq \emptyset$ and $h(U_1) \cap U_2 \neq \emptyset$ for all $g \in G \setminus G_1$ and $h \in H \setminus H_1$. \par 
			The assertion holds if there exists  $a \in A$ such that $a \in A_1 \cap g^{-1}(A_2)$ and $\phi(g,a)  \notin H_1$ for some $g \in G \setminus G_1$ . Since then $S^\phi_g(A_1 \times U_1) \cap (A_2 \times U_2) \neq \emptyset$, and so $(A \times X, S^\phi_G)$ is topologically transitive.\par 
			 But if $\phi(g,a) \in H_1$, for all  $a \in A_1 \cap g^{-1}(A_2)$ with $g \in G \setminus G_1$, then $\{\phi(g,a): a \in A_1 \cap g^{-1}(A_2), g \in G \setminus G_1\}$ is bounded as $H_1$ is compact. This contradicts the assumption. 
		\end{proof}	
	
	\begin{Remark}
		We note that the theorem above holds even when $G$ and $H$ are monoids giving a skew product semiflow $(A \times X, S^\phi_G)$.
	\end{Remark}

	 In the article we are interested in transitivity property of skew product with non-compact fiber. To proceed in this direction first we will characterize transitivity of skew products with locally compact fibers.		
		
		\begin{Theorem}\label{trans1}
			Let the semicascade $(A,f)$ be minimal, where $A$ is a compact metric space. Also, let $(X,d)$ be a locally compact metric space and $H$ be any topological subgroup of the homeomorphism group on $X$ with composition as the group operation, endowed with the compact open topology giving the flow $(X,H)$.\par  
			Suppose that $\phi:\mathbb{Z}_+ \times A\to H$, defined by  $\phi(n,a)=\phi_n(a):=\phi(f^{n-1}(a))\circ  \dots  \circ \phi(a)$ for all $n \in \mathbb{Z}_+$ is a cocycle defining the skew product $F:A\times X\to A\times X$ as $F(a,x)=(f(a),\phi(a)x)$ for every $(a,x)\in A\times X$. \par
			 If there exists some $a\in A$ such that the set $\{\phi(f^n(a)):n\in \mathbb{Z_+}\}$ is a commuting class of maps and the sequence of maps $(\phi_n(a))  $ is weakly mixing, then the skew product $(A\times X,F)$ is topologically transitive.
		\end{Theorem}
		
		\begin{proof}
			Let $A_1,A_2$ be opene in $A$ and $U,V$ opene in $X.$ For  $a \in A$ as in assumption, there exists $p,q\in \mathbb{N}$ such that $f^p(a)\in A_1$ and $f^{p+q}(a)\in A_2.$ Since $f^{p+q}(a)$ is an almost periodic point, the set $N(f^{p+q}(a),A_2)$ is syndetic with bounded gap $M$. Consequently the set $\{n\in \mathbb{Z}: n>q,\  f^{n+p}(a)\in A_2\}$ is also a syndetic set with bounded gap $M.$\par 
			Let $V^{\prime}=\phi(f^{p-1}(a))\circ\dots\circ\phi(a)V$, then $V^{\prime}$ is an opene subset of $X.$ We claim that $N_{(\phi_n(a))}(U,V^{\prime})$ is thick.\par
			Note that since $(\phi_n(a))$ is weakly mixing, from Theorem \ref{hereditarily}, $(\phi_n(a))$ is hereditarily transitive with respect to some sequence $(n_k)$. Now for every $m\in \mathbb{N}$ as $\phi_m(A)$ is compact, the sequence $(\phi_m(f^{n_k}(a)))$ has a convergent subsequence in $\phi_m(A).$ \par Let $l\in \mathbb{N}$. Then for every $m\in \{1,\dots,l\}$, we get a common subsequence $(r_k)$ of $(n_k)$ such that $(\phi_m(f^{r_k}(a)))$ converges to $g_m$(say). Again using Theorem \ref{hereditarily}, for this subsequence $(r_k)$ we can conclude that $(\phi_{r_k}(a))$ is weakly mixing.\par 
			
			Choose $\varepsilon>0$ and a point $y\in V^{\prime}$ such that $B(y,2\varepsilon)\subset V^{\prime}$ and also take $W=B(y,\varepsilon).$ Since $X$ is locally compact, there exist opene sets $W_m$, with $m\in \{1,\dots,l\}$ such that $\overline{W_m}$ is compact and $\overline{W_m}\subset g_m^{-1}(W)$. Then there exists $k_0\in \mathbb{N}$ such that $d(\phi_m(f^{r_k}(a))(z),g_m(z))<\varepsilon$ for every $z\in \overline{W_m}$, $k\geq k_0$ and $m\in \{1,\dots,l\}.$
			Now we have
			$$N_{(\phi_n(a))}(U,W)\bigcap (\displaystyle{\bigcap_{m=1}^l}N_{(\phi_n(a))}(U,W_m))\neq \emptyset$$
			We can choose $r_k\in N_{(\phi_n(a))}(U,W)\bigcap (\displaystyle{\bigcap_{m=1}^l}N_{(\phi_n(a))}(U,W_m))$ with $k>k_0.$ Then for every $m\in \{1,\dots,l\}$, there exists $u_{m}\in U$ such that $\phi_{r_k}(a)(u_{m})\in W_m.$ 
			
		Consequently, we have 		
		$ d(\phi_m(f^{r_k}(a))((\phi_{r_k}(a)(u_{m})),\ y)\ <$ 
		
		 $ d(\phi_m(f^{r_k}(a))(\phi_{r_k}(a)(u_{m})),$ $g_{m}(\phi_{r_k}(a)(u_{m})))+d(g_{m}(\phi_{r_k}(a)(u_{m})),y)<2\varepsilon.$\par
		 
		  Hence $\phi_{r_k+m}(a)(u_{m})\in V^{\prime}.$ Consequently, $r_k+m\in N_{(\phi_n(a))}(U,V^{\prime})$ for every $m\in \{0,\dots,l\}$ and so $N_{(\phi_n(a))}(U,V^{\prime})$ is thick.\par
			Now clearly there exists some $n_0\in \{n\in \mathbb{N}: f^{n+p}(a)\in A_2\}$ such that $n_0+p\in N_{(\phi_n(a))}(U,V^{\prime}).$ Consequently, we can choose $u\in U$ such that $\phi_{n_0}(f^p(a))(u)\in V.$ Hence $F^{n_0}(f^p(a),u)\in  A_2\times V.$ Therefore $(A\times X,F)$ is topologically transitive.			
		\end{proof}

Now the natural question is can you characterize a transitive product in this way, where the fiber is not locally compact. In the next section, we will see such a case when the fiber is an infinite dimensional Banach space. For particular cases we can have that the assertion in Theorem \ref{trans1} holds  true even when $X$ is not locally compact.

\subsection{Skew Product (semi)cascades}:	

In what follows we are particularly interested in $(\mathbb{Z}_+)\mathbb{Z}-$ cocycles defining skew product (semi)cascades.

			 We accordingly modify Definition \ref{skew prod gen}.
			
			\begin{Definition}\label{skew prod def gen2}
				Let $(A,f)$ and $(X,T)$ be  (semi)cascades, where $A$ and $X$ are  Hausdorff spaces.  Then for any continuous function $\tilde{h} : A \to (\mathbb{Z}_+)\mathbb{Z}$, we define a \emph{cocycle} $h: (\mathbb{Z}_+)\mathbb{Z} \times A \to (\mathbb{Z}_+)\mathbb{Z}$ as:
				\begin{center}
					$h(n,a) = \begin{cases} \sum \limits_{k=0}^{n-1}  \tilde{h}(f^k(a)) \ \ \text{if} \ \ n \geq 1,\\
						0 \ \ \text{if} \ \ n=0,\\
						-\sum \limits_{k=1}^{-n}  \tilde{h}(f^{-k}(a)) \ \ \text{if} \ \ n <0. \end{cases}$
				\end{center}
				
				 satisfying the cocycle identities.\par
				
				Define  $F:A\times X \to A\times X$ by $F(a,x)=(f(a),T^{h(1,a)}(x))$ for every $a\in A$ and $x\in X.$\par 
				Then $(X,F)$ is a \emph{skew product (semi)cascade}.
					\end{Definition}
		
\begin{Remark}
	 For every $n\in \mathbb{Z}$, the $n$th iterates are defined by $F^n(a,x)=(f^n(a),T^{h(n,a)}(x))$ for every $(a,x)\in A\times X$.
\end{Remark}

We note that a coboundary here is a map $g:  A \to (\mathbb{Z}_+)\mathbb{Z}$ for which $h(n,a) = g(f^n(a)) - g(a),$ for every $a \in A$ and $n \in   (\mathbb{Z}_+)\mathbb{Z}$.

\begin{Theorem}\label{coboundary}
	Let the cascade $(A,f)$ be minimal, where $A$ is a compact metric space. Also let $X$ be any complete metric space and $T:X\to X$ be a homeomorphism giving the cascade $(X,T).$ \par
	
	Let $h:\mathbb{Z} \times A \to \mathbb{Z}$ be the cocycle as Definition \ref{skew prod def gen2}, giving the skew product $(A \times X ,F)$ with 
		$$F(a,x)=(f(a), T^{h(1,a)}x) \ \text{for every} \ (a,x)\in A\times X.$$
		  Then  the following conditions are equivalent:
	\begin{enumerate}
		\item The cocycle $h$ is coboundary i.e., there exists a continuous function $g:A\to \mathbb{Z}$ such that $h(n,a)=g(f^n(a))-g(a)$ for every $n\in \mathbb{Z}$ and $a\in A.$
		\item There exists $a\in A$ such that the set $\{h(n,a):n\in \mathbb{Z}\}$ is bounded.
	\end{enumerate}
\end{Theorem}

\begin{proof}
	We first assume that the cocycle $h$ is a coboundary. Then there exists a continuous function $g:A\to \mathbb{Z}$ such that $h(n,a)=g(f^n(a))-g(a)$ for every $n\in \mathbb{Z}$ and $a\in A.$ Since $A$ is compact, $g(A)$ is finite. Consequently, the set $\{h(n,a):n\in \mathbb{Z}\}$ is bounded for every $a\in A.$\par
	Conversely, suppose that there exists some $a\in A$ such that the set $H = \{h(n,a):n\in \mathbb{Z}\}$ is bounded and hence finite. 
	
	If $H =\{0\}$ then vacuously $h$ is a coboundary. 
	
	Else, let us consider a function $G:A\times \mathbb{Z} \to A\times \mathbb{Z}$ by $G(b,r)=(f(b),h(1,b)+r)$ for every $(b,r)\in A\times \mathbb{Z}.$ Clearly for every $n\in \mathbb{Z}$, $G^n(a,0)=(f^n(a),h(n,a))$. Now from compactness of $f(A),$ and finiteness of $H$ we can conclude that $\overline{\mathcal{O}_G(a,0)}$ is  compact in $A\times \mathbb{Z}.$ Hence there exists a minimal set $M \subset \overline{\mathcal{O}_G(a,0)}$, which gives a minimal cascade $(M,G)$.\par
	Let $(a^*,r) \in M$. Then since $O_f(a^*)$ is dense in $A$, we have $\pi_1(M)=A$. 
	We claim that for every $b\in A,$ $\pi_1^{-1}(b)\cap M$ is a singleton set. On the contrary suppose that there is some $b\in A$ such that $\pi_1^{-1}(b)\cap M$ contains more than one point. That means $(b,m),(b,m+r)\in M$ for some $m,r(\neq 0)\in \mathbb{Z}.$ Let us consider the translation map $T_r:A\times \mathbb{Z} \to A\times \mathbb{Z}$, defined by $T_r(c,n)=(f(c),n+r)$ for every $(c,n)\in A\times \mathbb{Z}.$ Since $M$ is minimal, $(b,m+r)\in \overline{\mathcal{O}_G(b,m)}$. Then one can easily check that $(b,m+2r)\in T_r(\overline{\mathcal{O}_G(b,m)})\subset \overline{\mathcal{O}_G(b,m+r)}\subset M.$ In this way by induction we can prove that $(m+kr)\in M$ for every $k\in \mathbb{N}.$ But this contradicts the fact that $H$ is finite. \par  
	Define $g: A \to \mathbb{Z}$ such that $g(b)$ is the unique integer for which $(b,g(b)) \in M$. Then $G(b,g(b)) = (f(b),h(1,b)+g(b)) \in M$. But we also have $(f(b), g(f(b)) \in M$ and since each slice is a singleton, we must have $h(1,b) = g(f(b)) - g(b)$ for every $b \in A$. 
	Hence, $h(n,b) = \sum \limits_{i=0}^{n-1} h(1,f^i(b)) = \sum \limits_{i=0}^{n-1} \{g(f^{i+1}(b)) - g(f^i(b))\} = g(f^n(b)) - g(b)$ for $n > 0$. Also for $n<0$, $h(n,b) = -\sum \limits_{i=1}^{-n} h(1,f^{-i}(b))=\sum \limits_{i=1}^{-n} \{g(f^{-i}(b))-g(f^{-i+1}(b))\}=g(f^n(b)) - g(b)$. Thus, the cocycle $h$ is  a coboundary.	
\end{proof}

\begin{Theorem}
	Let $(A,f)$ be a minimal cascade, where $A$ is a compact metric space and $(X,T)$ be a cascade, where $X$ is a complete metric space.\par  Let $h:\mathbb{Z} \times A \to \mathbb{Z}$ be the cocycle as Definition \ref{skew prod def gen2}, giving the skew product $(A \times X ,F)$ with 
	$F(a,x)=(f(a), T^{h(1,a)}x) \ \text{for every} \ (a,x)\in A\times X.$
	Further, assume that
	
	(1)  $(X,T)$ is weakly mixing 
	
	(2) there exists   $a \in A$ for which the set $\{h(n,a):n\in \mathbb{Z}\}$ is unbounded.
	
	Then the skew product cascade $(A\times X, F)$ is topologically transitive.
\end{Theorem}
\begin{proof}
	Let $A_1,A_2$ be opene in $A$ and $U,V$ opene in $X.$ From minimality of $(A,f)$, there exists $p,q\in \mathbb{Z}$ such that $f^p(a)\in A_1$ and $f^{p+q}(a)\in A_2.$ Also, Since $f^{p+q}(a)$ is an almost periodic point, the set $N(f^{p+q}(a),A_2)$ is syndetic with bounded gap $M$. Consequently the set $\{n\in \mathbb{Z}: f^{n+p}(a)\in A_2\}$ is also a syndetic set with bounded gap $M.$ \par 
	On the other hand $\tilde{h}(A)$ is compact, since $\tilde{h}$ is continuous. Hence there exists $N\in \mathbb{N}$ such that $|h(n+1,b)-h(n,b)|< N$ for every $n\in \mathbb{Z}$ and $b\in A.$ Also, as $(X,T)$ is weakly mixing, we have $N(U,V)\cap \mathbb{Z_+}$ and $N(U,V)\cap \mathbb{Z_-}$ both are thick sets. Now we have following cases:\par 
	We first take the case when the set $\{h(n,a):n\in \mathbb{Z}_+\}$ is unbounded above. From definition of $h$, one can observe that the set $\{h(n,f^p(a)):n\in \mathbb{N}\}$ is also unbounded above. Let us take $r=h(1,f^p(a))+3MN$. Since $N(U,V)\cap \mathbb{Z_+}$ is thick, there exists some $m\in \mathbb{N}$ such that $\{m,m+1,\dots, m+r\}\subset N(U,V).$ Consider $t=m+h(1,f^p(a))+MN+1.$ Then clearly $h(1,f^p(a))< t-MN$ and $t-MN,
	t-MN+1,\dots,t,t+1,\dots,t+MN \in N(U,V).$ Now as $|h(n+1,f^p(a))-h(n,f^p(a))|< N$ for every $n\in \mathbb{N}$, and the set $\{h(n,f^p(a)):n\in \mathbb{N}\}$ is unbounded above, so there is
	some $n_0\in \mathbb{N}$ such that $h(n_0,f^p(a))\in [t,t+N]\cap \mathbb{N}$
	and consequently, $h(n_0,f^p(a)),h(n_0+1,f^p(a)),\dots, h(n_0+M,f^p(a))\in
	[t-MN,t+MN]\cap \mathbb{N}$. Then we get some $r\in \{n\in \mathbb{N}:f^{n+p}(a)\in A_2\}$ such that $h(r,f^p(a))\in [t-MN,t+MN]\cap
	\mathbb{N}$. Therefore $F^r(f^p(a),u)\in A_2\times V$ for some $u\in U.$
	This implies that $(A\times X, F)$ is topologically transitive.
	
	 In a similar way one can prove the other cases also.
 \end{proof}
 
 \begin{Remark}
 	From Theorem \ref{coboundary}, we can conclude that the above theorem is also true if the cocycle $h$ is not a coboundary.
 \end{Remark}
 
\bigskip

 \section{Transitivity properties of skew products for linear operators}
 
 In this section we follow up with our study for skew products in a linear setting. Since transitivity properties for linear operators are often expressed in terms of hypercyclicity criterion, most of our results of the previous section can be presented in more particular form in the linear setting. 
 
  We review the following definition of skew product from \cite{bayart} and this is an example of skew product, where the fiber is not locally compact. They consider the base of the skew product to be ergodic. In this section we assume that $\mu$ is a Borel probability measure with full support.
	
 	\begin{Definition}\label{skew def}\cite{bayart}
		Let $(A,f)$ be a semicascade, where $A$ is a compact metric space and $h:A\to \mathbb{C}$ be a continuous function. Also, let $X$ be an infinite dimensional separable complex Banach space and $T\in \mathcal{L}(X)$. Then the skew product $P:A\times X \to A\times X$ is defined as $P(a,x)=(f(a),h(a)Tx)$ and the $n$th iterates of $P$ is\par 
	\noindent	$P^n(a,x)=(f^n(a),h_n(a)T^n(a))$, where $h_n(a):=h(f^{n-1}(a))h(f^{n-2}(a))\dots h(a). $
		 
	\end{Definition}
	\begin{Remark}
		For each $i\in \mathbb{Z_+},$ let $H_i = \{cT^i: c\in \mathbb{C}\}$. Clearly $H_i$ is a topological monoid endowed with the pointwise convergence topology. One can easily check that if we take $H=\displaystyle{\coprod_{i\in \mathbb{Z}_+}}\{cT^i: c\in \mathbb{C}\}$, the topological disjoint union of $H_i,$ where $i\in \mathbb{Z_+}$ and the cocycle $\phi: \mathbb{Z}_+ \times A \to H$ defined as $\phi(n,a)=h_n(a)T^n$ then from Definition \ref{skew prod  gen}, we get  Definition \ref{skew def}, i.e $P^n(a,x)=(f^n(a),\phi(n,a)x)$ for every $n\in \mathbb{Z}_+, a\in A, x\in X.$ Note that here $H$ is a topological monoid.
	\end{Remark}

	We recall the following theorem from \cite{bayart}, where the authors discussed certain hypercyclicity type criterion for the skew product defined above. 
	
	\begin{Theorem}\cite{bayart}\label{hc}
		Let the system $(A,\mu,f)$ be ergodic, where $A$ is a compact metric space and $\mu$ is a Borel probability measure on $A$. Let $h:A\to \mathbb{C}$ be a continuous function with $\gamma= \displaystyle{\int_A} \log |h| d\mu$ finite. \par 
		Also let $X$ be an infinite dimensional separable complex Banach space and $T\in \mathcal{L}(X)$. Then the skew product system  $(A\times X, P)$,  is topological transitive if there exist two dense subsets $D_1, D_2$ of $X$ and a sequence of mappings $S_{n}:D_2 \to X$ such that\par 
		$(i)$ $\limsup_n || T^{n} x||^{\frac{1}{n}}< e^{-\gamma}$ for every $x\in D_1.$ \par 
		$(ii)$ $\limsup_n || S_{n} y||^{\frac{1}{n}}< e^{\gamma}$ for every $y\in D_2.$ \par
		$(iii)$ $|| T^{n}S_{n}y - y|| \to 0$ as $k\to \infty$ for every $y\in D_2.$ 
		
	\end{Theorem}

In linear dynamics usually hypercyclicity criterion are considered with respect to some arbitrary sequence $(n_k)$. In the following theorem we will prove topological transitivity of the above mentioned skew product, where we assume the conditions $(i)-(iii)$ of Theorem \ref{hc} with respect to some sequence $(n_k)$ instead of the full sequence $(n).$ To prove this we need to take the base of the skew product to be ergodic and minimal. 
	\begin{Theorem}\label{mini-weak}
		Let $(A, f)$ be a minimal semicascade, where $A$ is a compact metric space and also, the system $(A,\mu,f)$ be ergodic, where $\mu$ is a Borel probability measure on $A$. Let $h:A\to \mathbb{C}$ be a continuous function with $\gamma= \displaystyle{\int_A} \log |h| d\mu$ finite.\par 
		Suppose that $X$ is an infinite dimensional separable complex Banach space and $T\in \mathcal{L}(X)$. Then the skew product system  $(A\times X, P)$,  is topologically transitive if there exist two dense subsets $D_1, D_2$ of $X$, an increasing sequence $(n_k)$ of positive integers and a sequence of mappings $S_{n_k}:D_2 \to X$ such that\par 
		$(i)$ $\limsup_{k} || T^{n_k} x||^{\frac{1}{n_k}}< e^{-\gamma}$ for every $x\in D_1.$ \par 
		$(ii)$ $\limsup_{k} || S_{n_k} y||^{\frac{1}{n_k}}< e^{\gamma}$ for every $y\in D_2.$ \par 
		 $(iii)$ $|| T^{n_k}S_{n_k}y - y|| \to 0$ as $k\to \infty$ for every $y\in D_2.$\\
	 Moreover the set $\{a\in A:\  (a,x)$ is a transitive point of $(A\times X, P)$ for some $x\in X\}$ has measure $1.$
	\end{Theorem}
	
	\begin{proof}
		Let us consider the set 
		
		\begin{center}
			$C=\{a\in A: \frac{1}{n}\displaystyle{\sum_{j=0}^{n-1}} \log|h(f^j(a))|\to \displaystyle{\int_A} \log |h| d\mu \}.$
		\end{center}
		From Theorem \ref{Birkhoff}, we have $\mu(C)=1.$ We claim that for every $a\in C$, there exists some $x\in X$ such that $(a,x)$ is a transitive point for $(A\times X, P).$ For every pair of opene sets $B\subset A$, $V\subset X$ and for every $a\in C$ we define the following set
		\begin{center}
			$E(a,B,V)=\{x\in X:\  \text{there exists some}\ n\in \mathbb{N}\ \text{such that}\ f^n(a)\in B\ \text{and}\ h_n(a)T^n(x)\in V\}$
		\end{center}
		Clearly $E(a,B,V)$ is an open subset of $X.$ Now we will show that $E(a,B,V)$ is dense in $X.$ Let $U$ be any opene subset of $X.$ Firstly, there exists $p\in \mathbb{N}$ such that $f^p(a)\in B.$ Since $f^{p}(a)$ is an almost periodic point, the set $N(f^{p}(a),B)$ is syndetic. Consequently, the set $\{n>p: f^n(a)\in B\}$ is also syndetic. Now we only need to prove that the set $\{n\in \mathbb{N}:h_n(a)T^n(U)\cap V \neq \emptyset\}$ is thick. \par
		
		Let $r\in \mathbb{N}.$ From the given assumption, let us choose $z\in U\cap D_1$ and for every $i\in \{0,1,\dots, r\}$, $y_i\in T^{-i}(V)\cap D_2.$ Also we may choose $\varepsilon> 0$ small enough such that $B(z,\varepsilon)\subset U$ and $B(y_i,2\varepsilon)\subset T^{-i}(V)$ for each $i\in \{0,1,\dots,r\}.$ Again from the given conditions we can say that there exists some $\eta>0$ and $K\in \mathbb{N}$ such that $||T^{n_k}(z)||^{\frac{1}{n_k}}\leq e^{-\gamma-\frac{\eta}{2}}$, $||S_{n_k}(y_i)||^{\frac{1}{n_k}}\leq e^{\gamma-\frac{\eta}{2}}$ and $||T^{n_k}S_{n_k}(y_i)-y_i||<\varepsilon$ for every $k>K$, where $i=0,1,\dots,r.$ On the other hand as $a\in C$, there exists $N\in \mathbb{N}$ such that for every $n\geq N$ 
		\begin{center}
			
			$e^{n(\gamma-\frac{\eta}{4})}<h_n(a)<e^{n(\gamma+\frac{\eta}{4})}.$ 
		\end{center}
		One can easily check that for $k$ large enough, for every $i\in \{0.\dots,r\}$ we can choose  $z_i=z+h_{n_k+i}(a)^{-1}S_{n_k}(y_i)\in U$ such that $h_{n_k+i}(a)T^{n_k+i}(z_i)=T^i(h_{n_k+i}(a)T^{n_k}(z_i))\in V$ as \begin{center}
			$||h_{n_k+i}(a)T^{n_k}(z_i)-y_i||=||h_{n_k+i}(a)T^{n_k}(z)||+||T^{n_k}S_{n_k}(y_i)-y_i||<2\varepsilon.$
		\end{center}
		Hence the set $\{n\in \mathbb{N}:h_n(a)T^n(U)\cap V \neq \emptyset\}$ is thick. Then there exists some $n\in \mathbb{N}$ and $u\in U$ such that $f^n(a)\in B$ and $h_n(a)T^n(u)\in V$, which implies that $E(a,B,V)$ is dense. Then using Baire Category theorem, we can say that $\bigcap E(a,B,V)$ is non-empty. Consequently, there exists some $x\in X$ such that $(a,x)$ is a transitive point of $P.$ 				 
	\end{proof}
	
	\begin{Remark}
				In Theorem \ref{mini-weak}, we can take the system $(A,\mu, f)$ to be strictly ergodic which is equivalent to minimality. Hence Theorem 2.9 of \cite{bayart} is true for some sequence $(n_k)$ also, making Theorem \ref{mini-weak} stronger than Theorem 2.9 of \cite{bayart}. We illustrate this with an example:

			\begin{Example}
				Let $\gamma$ and $\epsilon$ be two positive number such that $(\gamma-\varepsilon)>0$. First of all we consider a sequence $(r_k)$ of positive integers such that $r_1=1$ and $r_{k+1}=4r_k+3$ for every $k\in \mathbb{N}.$ Then we define a sequence $(w_n)$ by
				\begin{center}
					$w_{n}=
					\begin{cases}
						e^{(r_k-i+1)(\gamma-\varepsilon)}& \text{when}\ n=(r_k+i)\ \text{for}\  i=1,\dots, (r_k+1)\\
						e^{(2r_k+3-i)2\gamma}& \text{when}\ n=(2r_k+i)\ \text{for}\ i=2,\dots, (2r_k+3) 
					\end{cases}$
				\end{center}
				Now let $\ell^2(w)=\{(x_n): \displaystyle{\sum_{n=1}^{\infty}}|x_n|^2 w_n^2 < \infty\}$ be a weighted $\ell^2$-space. Since $\sup_n\frac{w_n}{w_{n+1}} < \infty,$ the backward shift operator $B,$ defined by $B(x_1,x_2,x_3,\dots)=(x_2,x_3, \dots)$ is continuous on $\ell^2(w)$ \cite{linear chaos}. \par 
				Let the system $(A,\mu,f)$ be strictly ergodic, where $A$ is a compact metric space. Now we consider the skew product $P:A\times \ell^2(w)\to A\times \ell^2(w)$ by $P(a,x)=(f(a),h(a)B(x)),$ where $h:A\to \mathbb{C}$ is defined by $h(a)=e^{\gamma}$ for every $a\in A.$\par   
				Let us take $D_1=D_2=$ space of all sequences, where all but finitely many terms of the sequence are zero and a mapping $S: D_2\to D_2$, defined by $S(x_1,x_2,\dots)=(0,x_1,x_2,\dots)$. Let $x\in D_2$. Then there exist some $k_0\in \mathbb{N}$ and $M>0$ such that for every $k\geq k_0$ we have $||S^{r_k}(x)||^2=\displaystyle{\sum_{i=1}^{k_0}}w_{r_k+i}^2 |x_i|^2= \displaystyle{\sum_{i=1}^{k_0}} e^{2(r_k-i+1)(\gamma-\varepsilon)}|x_i|^2 \leq M e^{2r_k(\gamma-\varepsilon)}.$ Therefore $\limsup_{k}||S^{r_k}(x)||^{\frac{1}{r_k}}< e^{\gamma}.$ Now it is clear that $B$ satisfies all the conditions $(i)-(iii)$  of Theorem \ref{mini-weak}.
				However it does not satisfy the conditions of Theorem 2.9 of \cite{bayart}. \par 
				Note that for every $x\in D_2$ there exists some $k_1\in \mathbb{N}$ and $M_1>0$ such that for every $k>k_1$ we have $||S^{2r_k+1}(x)||^2=\displaystyle{\sum_{i=1}^{k_0}}w_{2r_k+1+i}^2 |x_i|^2= \displaystyle{\sum_{i=1}^{k_0}} e^{4\gamma(2r_k-i+2)}|x_i|^2 = M_1 e^{4\gamma (2r_k+1)}.$ Therefore $\limsup_{n}||S^{n}(x)||^{\frac{1}{n}}>e^{\gamma}.$\par 
				 Here $P$ is topologically transitive.
			\end{Example}
			\end{Remark}

				\begin{Theorem} \label{nlc}
						Let $(A,f)$ be a minimal semicascade, with $A$ is a compact metric space and $\mu$ be an ergodic probability measure on $A$ for $f.$ Let $h:A\to \mathbb{C}$ be a continuous function with $\gamma= \displaystyle{\int_A} \log |h| d\mu$ be finite. Then the skew product $P$ is topological transitive if there exists $a\in A$ such that the sequence operators $(h_n(a)T^n)$ is weakly mixing.
					\end{Theorem}
				\begin{proof}
					Clearly the sequence operators $(h_n(a)T^n)$ satisfies Universality Criterion. If we carefully observe the proof of Theorem \ref{mini-weak}, then from Universality Criterion of $(h_n(a)T^n)$, we can prove that for every pair of opene sets $U,V$ of $X,$ the set $\{n\in \mathbb{N}: h_n(a)T^n(U)\cap V \neq \emptyset\}$ is a thick set. Then same as Theorem \ref{mini-weak}, we can conclude that $P$ is topologically transitive.
					\end{proof}
			\begin{Remark}
				From the above theorem we can say that the implications of  Theorem \ref{trans1} is also true for this skew product $P$, where the fiber is not locally compact. 
			\end{Remark}

	\vspace{0.05cm}

	It is well known  that weakly mixing and hypercyclicity criterion are equivalent in linear structure. Then one can naturally ask what happens when the concept of skew products comes to the picture. Now we discuss the relation between weakly mixing property of skew products with the conditions discussed in Theorem \ref{mini-weak}. For an infinite dimensional Banach space $X,$ we consider the Euclidean norm in the product space $X\times X$ i.e., $||(x,y)||=\sqrt{||x||^2+||y||^2}$ for every $(x,y)\in X\times X.$
	
		\begin{Theorem} \label{wmsp}
			
			Let the system $(A,\mu,f)$ be weakly mixing, where $A$ is a compact metric space and $\mu$ is a probability measure on $A$. Let $h:A\to \mathbb{C}$ be a continuous function with $\gamma= \displaystyle{\int_A} \log |h| d\mu$ finite. \par 
			Suppose that $X$ is an infinite dimensional separable complex Banach space and $T\in \mathcal{L}(X)$. Then the skew product system  $(A\times X, P)$ is weakly mixing if there exist two dense subsets $D_1, D_2$ of $X$, a syndetic sequence $(n_k)$ of positive integers and a sequence of mappings $S_{n_k}:D_2 \to X$ such that\par 
			$(i)$ $\limsup_{k} || T^{n_k} x||^{\frac{1}{n_k}}< e^{-\gamma}$ for every $x\in D_1.$ \par 
			$(ii)$ $\limsup_{k} || S_{n_k} y||^{\frac{1}{n_k}}< e^{\gamma}$ for every $y\in D_2.$ \par
			$(iii)$ $|| T^{n_k}S_{n_k}y - y|| \to 0$ as $k\to \infty$ for every $y\in D_2.$ \\
			Moreover the set $\{(a,b)\in A\times A:\  (a,x,b,y)\in Trans_{P\times P} \ \text{for some}\ x,y\in X\}$ has measure $1$ with respect to the product measure induced from $\mu.$
	\end{Theorem}
	\begin{proof}
		One can easily check that $T\times T$ satisfies given conditions $(i)-(iii)$. Also, since $f$ is weakly mixing in $(A,\mu)$, $f\times f$ is ergodic with respect to the product measure induced from $\mu$. Then the proof follows similarly as in Theorem \ref{mini-weak}. 
	\end{proof}

		\begin{Theorem}\label{weakly mixing}
			Let the system $(A,\mu,f)$ be ergodic, where $A$ is a compact metric space and  and $\mu$ is a probability measure on $A$. Let $h:A\to \mathbb{C}$ be a continuous function with $\gamma= \displaystyle{\int_A} \log |h| d\mu$ finite.\par 
			 Suppose that the skew-product system $(A\times X, P)$ is weakly mixing and the set $\{(a,b)\in A\times A: (a,x,b,y)\in Trans_{P\times P}\ \text{for some}\ x,y\in X \}$ has positive measure with respect to the product measure induced by $\mu$. Then there exist an increasing sequence of positive integers $(n_k)$, two dense subsets $D_1, D_2$ of $X$ and a sequence of mappings $S_{n_k}:D_2 \to X$ for which $T$ satisfies the following conditions:\par 
		$(i)$ $\limsup_k || T^{n_k} x||^{\frac{1}{n_k}}\leq e^{-\gamma}$ for every $x\in D_1.$ \par 
		$(ii)$ $\limsup_k || S_{n_k} y||^{\frac{1}{n_k}}\leq e^{\gamma}$ for every $y\in D_2.$ \par
		$(iii)$ $|| T^{n_k}S_{n_k}y - y|| \to 0$ as $k\to \infty$ for every $y\in D_2.$ 
		
	\end{Theorem}
	\begin{proof}
		Let us consider the set 
		\begin{center}
			$C=\{b\in A: \frac{1}{n}\displaystyle{\sum_{j=0}^{n-1}} \log|h(f^j(b))|\to \displaystyle{\int_A} \log |h| d\mu \}.$
		\end{center}
		Using Theorem \ref{Birkhoff}, we can say that $C\times C$ has measure $1$ with respect to the product measure induced from $\mu.$ Then there exist $a,b\in C$ such that $(a,x,b,y)$ is a transitive point of $P\times P$ for some $x,y\in X$.\par 
		First of all we claim that for every $n\in \mathbb{N},$ $(a,x,b,h_n(b)T^ny)\in Trans_{P \times P}.$ Note that as $(a,x,b,y)\in Trans_{P \times P}$, $h_n(a),h_n(b)\neq 0$ for every $n\in \mathbb{N}$. Also, Since $P$ is topologically transitive, then using linearity of $T$, one can easily check that $T(X)$ is dense in $X$. Let $O_1,O_2$ be two opene subsets of $A$ and $U_1,U_2$ be two opene subsets of $X.$ Let $n\in \mathbb{N}.$ Now from continuity of $T^n,$ we can say that $T^{-n}(h_n(b)^{-1}U_2)$ is an opene subset of $X$. Then there exists $k\in \mathbb{N}$ such that $P^k(a,x)\in O_1\times U_1$ and $P^k(b,y)\in O_2\times T^{-n}(h_n(b)^{-1}U_2).$ This implies that $h^k(b)T^k(h_n(b)T^ny)\in U_2.$ This proves the claim. \par 
		On the other hand, as $(a,x,b,y)\in Trans_{P\times P},$ it is clear that the set $\{h_n(b)T^n(y):n\in \mathbb{N}\}$ is dense in $X.$ So there exists $z_k\to 0$ such that $(a,x,b,z_k)\in Trans_{P\times P}$ for each $k\in \mathbb{N}.$ This implies that for each $k\in \mathbb{N},$ there exists $n_k$ such that $h_{n_k}(a)T^{n_k}x\to 0$ and $h^{n_k}(b)T^{n_k}(z_k)\to x.$ \par 
		We take $D_1=D_2=D=\{h_j(a)T^jx:j\in \mathbb{N}\}$ and define a sequence of mappings $S_{n_k}:D\to X$ by $S_{n_k}(h_j(a)T^jx)=h^{n_k}(b)h_j(a)T^jz_k$ for every $j\in \mathbb{N}.$ Then for every $j\in \mathbb{N},$ we have
		\bigskip
		\par
		$(i)$ $h_{n_k}(a)T^{n_k}(h_j(a)T^jx)=h_j(a)T^j(h_{n_k}(a)T^{n_k}x)\to 0.$\par 
		$(ii)$  $h^{n_k}(b)^{-1}S_{n_k}(h_j(a)T^jx)\to 0.$\par 
		$(iii)$ $T^{n_k}S_{n_k}(h_j(a)T^jx)=T^{n_k}(h_j(a)T^j(h^{n_k}(b)z_k))\to h_j(a)T^jx.$\par 
		\bigskip
		Since $a,b\in C$, for any $\varepsilon>0,$ there exists $N\in \mathbb{N}$ such that $e^{n(\gamma-\varepsilon)}<h_n(a)<e^{n(\gamma+\varepsilon)}$ and $e^{n(\gamma-\varepsilon)}<h_n(b)<e^{n(\gamma+\varepsilon)}$ for every $n\geq N.$ Then for every $u\in D$ we get the followings:
		\par
		\bigskip
		 
		$(i)$ There exists $K\in \mathbb{N}$ such that for every $k\geq K$, $||h_{n_k}(a) T^{n_k}(u)||<1$ 
		$\Rightarrow$ $||T^{n_k}(u)||<e^{-n_k(\gamma-\varepsilon)}$ 
		$\Rightarrow$ $||T^{n_k}(u)||^{\frac{1}{n_k}}<e^{-(\gamma-\varepsilon)}.$ Hence $\limsup_k ||T^{n_k}(u)||^{\frac{1}{n_k}}\leq e^{-\gamma}.$\par 
		$(ii)$ Similarly $\limsup_k||S_{n_k}(u)||^{\frac{1}{n_k}}\leq e^{\gamma}.$\par 
		$(iii)$ $|| T^{n_k}S_{n_k}u - u|| \to 0$ as $k\to \infty.$
		\bigskip
		\\ 
		This completes the proof.
	\end{proof}

Now from Theorem \ref{wmsp} and Theorem \ref{weakly mixing}, one can deduce the following:

\begin{Corollary}
	Let the system be $(A,\mu, f)$ be weakly mixing, where $A$ is a compact metric space and  and $\mu$ is a probability measure on $A$. Also, let $B_w$ be the weighted backward shift on $\ell^p(\mathbb{N})$ defined by $B_w(x_1,x_2,\dots)=(w_1x_2,w_2x_3,\dots)$. Consider the skew product $P:A\times \ell^p(\mathbb{N})\to A\times \ell^p(\mathbb{N})$ by $P(a,x)=(g(a),h(a)B_w(x)),$ where $h:A\to \mathbb{C}$ is a continuous function with $\gamma= \displaystyle{\int_A} \log |h| d\mu$ finite. Then  \par
	(i) If $\limsup_n(\displaystyle{\prod_{i=1}^{n}}w_i)^{\frac{1}{n}} > e^{-\gamma}$ then $P$ is weakly mixing and the set $\{(a,b)\in A\times A:\  (a,x,b,y)\in Trans_{P\times P} \ \text{for some}\ x,y\in X\}$ has measure $1$ with respect to the product measure induced from $\mu.$\par
	$(ii)$ If $\limsup_n(\displaystyle{\prod_{i=1}^{n}}w_i)^{\frac{1}{n}} < e^{-\gamma}$ then either $P$ is not weakly mixing or the set $\{(a,b)\in A\times A:\  (a,x,b,y)\in Trans_{P\times P} \ \text{for some}\ x,y\in X\}$ has measure $0$ with respect to the product measure induced from $\mu.$
	
\end{Corollary}

In the following example, we will show that converse of the Theorem \ref{weakly mixing} is not true. Also, the conditions $(i)-(iii)$ mentioned in Theorem \ref{mini-weak} is strictly stronger than weakly mixing property of the sequence of operators $(h_n(a)T^n)$.\par

\begin{Example}\label{ex1}

	Let the system be $(A,\mu, g)$ be ergodic but not strictly ergodic, where $A$ is a compact metric space and  and $\mu$ is a probability measure on $A$. Also we assume that the semicascade $(A,g)$ is not syndetically transitive (A semicascade $(A,g)$ is called syndetically transitive if for every pair of opene subsets $U$ and $V$ of $A$, the set $N(U,V)$ is syndetic). 
	\par
	Now since $(A,g)$ is not syndetically transitive, there exist two opene subsets $A_1$ and $A_2$ of $A$ such that $N_g(A_1,A_2)$ is not syndetic, which implies that $\mathbb{N}\setminus N_g(A_1,A_2)$ contains a thick set. Then for every $k\in \mathbb{N},$ we can choose $r_k$ such that $(r_k-k),\dots,r_k,\dots,(r_k+k)\}\notin N(A_1,A_2)$.\par
	Let $\gamma$ be any positive real number. Let us define a bounded sequence $(w_n)$ by\par
	\begin{center}
		
		$w_{n}=
		\begin{cases}
			1& \text{when}\ n=(r_k-i)\ \text{for}\  i=0,1,\dots, k-1\\
			e^{-2\gamma}& \text{when}\ n=(r_k+i)\  \text{for}\  i=1,\dots, k \\
			e^{-\gamma}& \text{elsewhere}
		\end{cases}$
	\end{center}
	
	Consequently, we have
	\begin{center}
		$\displaystyle{\prod_{i=1}^{n}}w_i=
		\begin{cases}
			e^{-\gamma(n-k+i)} & \text{when}\ n=(r_k\pm i)\ \text{for}\ i=0,1,\dots, k\\
			e^{-n\gamma}& elsewhere
		\end{cases}$
	\end{center}
	\vspace{0.10cm}
	Let $B_w$ be the weighted backward shift on $\ell^p(\mathbb{N})$ defined by $B_w(x_1,x_2,\dots)=(w_1x_2,w_2x_3,\dots)$. Now we consider the skew product $P:A\times \ell^p(\mathbb{N})\to A\times \ell^p(\mathbb{N})$ by $P(a,x)=(g(a),h(a)B_w(x)),$ where $h:A\to \mathbb{C}$ is defined by $h(a)=e^{\gamma}$ for every $a\in A.$\par
	Let $D_1=D_2=$ space all sequences, where all but finitely many terms of the sequence are zero. Let us consider a mapping $S: D_2\to D_2$, defined by $S(x_1,x_2,\dots)=(0,\frac{x_1}{\omega_1},\frac{x_2}{\omega_2},\dots)$. Then by taking $S_n=S^n, n\in \mathbb{N}$, one can easily check that for every $a\in A,$ $B_w$ satisfies all the $(i)-(iii)$ conditions of Theorem \ref{weakly mixing}. But $P$ is not topologically transitive.\par

	On the contrary suppose that $P$ is topologically transitive. Let $U$ and $V$ be the $\frac{1}{2}$-neighbourhood of $0$ and $e_1$ respectively in $\ell^p(\mathbb{N})$. Then there exists some $n$ such that $P^n(A_1\times U)\cap (A_2\times V)\neq \emptyset$ and so $n$ must be belong to $N_g(A_1,A_2).$ This implies that there is some $a\in A_1$ and $x\in U$ such that
	\begin{center}
		$||h^n(a)B_w^n(x)-e_1||\geq |\displaystyle{\prod_{i=1}^{n}}w_i e^{-n\gamma}x_n-1|>\frac{1}{2}$,
	\end{center}
	which is a contradiction. \par
	Moreover, $B_w$ does not satisfy the conditions $(i)-(iii)$ of Theorem \ref{mini-weak}, since if it satisfies these conditions then from Theorem 2.4 of \cite{bayart}, $P$ must be topologically transitive, which is again a contradiction. So the converse of Theorem \ref{weakly mixing} is not true.\par
	On the other hand one can see that for every $a\in A$, the sequence of operators $(h_{r_k}(a)B_w^{r_k})$ satisfies Universality Criterion i.e., weakly mixing. But still $P$ is not topologically transitive as the base $(A,g)$ is not minimal.

\end{Example}
	
	\begin{Theorem}
		Suppose that $P:A\times X \to A\times X$ is a skew-product defined by $P(a,x)=(f(a),h(a)Tx)$ as Definition \ref{skew def}, $\mu$ is a probability measure on $A$  and $f:A\to A$ is strictly ergodic. Let 
		\begin{center}
			$\gamma= \displaystyle{\int_A} \log |h| d\mu$
		\end{center} 
		be finite. We assume that $f$ is topologically mixing and there exist an increasing sequence of positive integers $(n_k)$ with $\sup_k (n_{k+1}-n_k)< \infty$, two dense subsets $D_1, D_2$ of $X$ and a sequence of mappings $S_{n_k}:D_2 \to X$ such that\par 
		$(i)$ $\limsup_k || T^{n_k} x||^{\frac{1}{n_k}}< e^{-\gamma}$ for every $x\in D_1.$ \par 
		$(ii)$ $\limsup_k || S_{n_k} y||^{\frac{1}{n_k}}< e^{\gamma}$ for every $y\in D_2.$ \par
		$(iii)$ $|| T^{n_k}S_{n_k}y - y|| \to 0$ as $k\to \infty$ for every $y\in D_2.$\\ Then $P$ is topologically mixing.	
	\end{Theorem}
	
	\begin{proof}
		Suppose that $M=\sup_k(n_{k+1}-n_k).$ Let $B\times U$ and $C\times V$ be two non-empty open subsets of $A\times X.$ Let us choose $x\in U\cap D_1$ and for every $i\in \{0,1,\dots, M\}$, $y_i\in T^{-i}(V)\cap D_2.$ Also we may choose $\varepsilon> 0$ small enough such that $B(x,\varepsilon)\subset U$ and $B(y_i,2\varepsilon)\subset T^{-i}(V)$ for each $i\in \{0,1,\dots,M\}.$ Now since $f$ is topologically mixing, there exists $N_1\in \mathbb{N}$ such that for every $n\geq N_1$ there is some $b_n\in B$ with $f^n(b_n)\in C.$ From the given conditions, we can say that there exists some $\eta>0$ and $K\in \mathbb{N}$ such that $||T^{n_k}(x)||^{\frac{1}{n_k}}\leq e^{-\gamma-\frac{\eta}{2}}$, $||S_{n_k}(y_i)||^{\frac{1}{n_k}}\leq e^{\gamma-\frac{\eta}{2}}$ and $||S_{n_k}(y_i)-y_i||<\varepsilon$ for every $k>K$, where $i=0,1,\dots,M.$ On the other hand using Theorem \ref{oxtoby}, we have there exists some $N_2\in \mathbb{N}$ such that for every $a\in A$ and for every $n\geq N_2,$ $e^{n(\gamma-\frac{\eta}{4})}<h_n(a)<e^{n(\gamma+\frac{\eta}{4})}.$\par 
		Now for every $n\geq n_1$, there is some $n_k$ and $0\leq i\leq M$ such that $n=n_k+i$ and we consider $y_n=x+h_n(b_n)^{-1}S_{n_k}(y_i).$ Let us take $N>{N_1,N_2,n_K}$ with $e^{-N\frac{\eta}{4}}<\varepsilon.$ Then for every $n>N,$ one can easily check that $y_n\in U$ and $h_n(b_n)T^n(y_n)=T^i(h_n(b_n)T^{n_k}(y_n))\in V$ as
		\begin{center}
			$||h_n(b_n)T^{n_k}(y_n)-y_i||=||h_n(b_n)T^{n_k}(x)||+||S_{n_k}(y_i)-y_i||<2\varepsilon.$
	\end{center}
		Hence for every $n>N,$ $(b_n,y_n)\in B\times U$ and $P^n(b_n,y_n)\in C\times V.$ Therefore $P$ is topologically mixing.		
	\end{proof}

\end{document}